\documentclass[11pt]{article}

\usepackage{amsmath,amssymb,amsthm,mathtools}
\usepackage[hidelinks]{hyperref}
\usepackage[margin=2cm]{geometry}

\theoremstyle{plain}
\newtheorem{theorem}{Theorem}[section]
\newtheorem{corollary}[theorem]{Corollary}
\newtheorem{lemma}[theorem]{Lemma}
\newtheorem{proposition}[theorem]{Proposition}
\theoremstyle{definition}
\newtheorem{definition}[theorem]{Definition}
\theoremstyle{remark}
\newtheorem{remark}[theorem]{Remark}

\newcommand{\R}{\mathbb{R}}

\newcommand{\Z}{\mathbb{Z}}
\newcommand{\N}{\mathbb{N}}

\newcommand{\T}{\mathbb{T}}

\DeclareMathOperator{\diam}{diam}

\begin{document}

\begin{center}
{\Large Uniform optimal-order Wasserstein quantisation}\\[0.6em]
{\large Maja Gw\'o\'zd\'z}\\
{\normalsize University of Zurich \& ETH Z\"{u}rich}\\
{\normalsize \texttt{mgwozdz@ethz.ch}}\\[0.8em]
\end{center}

\begin{abstract}
We address Steinerberger's Wasserstein transport problem on the cube
\(Q=[0,1]^d\). For every \(d\ge2\), we consider a dyadic digital sequence \((x_n)\subset Q\) and prove that every prefix
\(\{x_1,\dots,x_N\}\) admits an exact equal-mass transport partition at the
optimal scale. More precisely, for every \(N\in\N\), there exist
pairwise disjoint Borel sets \(A_1,\dots,A_N\subset Q\) such that
\[
\lambda_d(A_n)=\frac1N,\qquad
A_n\subset B(x_n,6\sqrt d\,N^{-1/d})\qquad(1\le n\le N),
\]
and \(\lambda_d\!\bigl(Q\setminus\bigcup_{n=1}^N A_n\bigr)=0\). In other terms,
every prefix of the sequence supports an exact transport allocation of Lebesgue
mass to its points with uniformly controlled radius \(O(N^{-1/d})\). By an
elementary partition criterion, this yields
\[
W_\infty\!\left(\frac1N\sum_{n=1}^N\delta_{x_n},\,\lambda_d\right)\le 6\sqrt d\,N^{-1/d}
\qquad(N\in\N).
\]
The bound holds for every \(1\le p\le\infty\). The exponent
\(1/d\) is optimal, so it gives the sharp uniform prefix rate on the cube.
The result settles Steinerberger's problem for all \(d\ge1\) and all
\(1\le p\le\infty\).
\end{abstract}

\medskip
\noindent\textbf{2020 Mathematics Subject Classification.} 49Q22, 11K38, 11K31.\\
\textbf{Keywords.} Wasserstein distance, optimal transport, transport partitions, quasi-Monte Carlo, digital sequences.

\section{Introduction}\label{sec:introduction}

\subsection{Steinerberger's problem}
In \cite[Problem 47]{SteinerbergerOpenProblems}, Steinerberger asks whether there exists a deterministic infinite sequence
\((x_n)\subset[0,1]^d\) whose empirical measures converge to Lebesgue measure at the optimal-order Wasserstein rate uniformly over all prefix lengths.
Let \(Q:=[0,1]^d\), and let us denote by \(\lambda_d\) the normalised Lebesgue measure on \(Q\). Suppose we are given a sequence \((x_n)\subset Q\). We define its empirical measures by
\[
\mu_N:=\frac1N\sum_{n=1}^N\delta_{x_n}.
\]
For Borel probability measures \(\mu,\nu\) on \(Q\), let \(\Pi(\mu,\nu)\)
denote the set of couplings of \(\mu\) and \(\nu\). For \(1\le p<\infty\),
define
\[
W_p(\mu,\nu):=\Bigl(\inf_{\pi\in\Pi(\mu,\nu)}\int_{Q\times Q}|x-y|^p\,\mathrm d\pi(x,y)\Bigr)^{1/p},
\]
and define
\[
W_\infty(\mu,\nu):=\inf_{\pi\in\Pi(\mu,\nu)}\operatorname*{ess\,sup}_{(x,y)\sim\pi}|x-y|.
\]
We also write \(W_p:=W_p^{(Q,|\cdot|)}\).
The question is whether, for fixed \(d\) and \(p\), it is possible to find a single infinite sequence \((x_n)\subset Q\) such that
\begin{equation}\label{eq:OP47}
\exists\,C_{p,d}>0\ \forall N\in\N:\qquad
W_p(\mu_N,\lambda_d)\le C_{p,d}\,N^{-1/d}.
\end{equation}
The important point here is the uniformity in \(N\). Once we have fixed the
sequence, every prefix satisfies the optimal-order bound. Steinerberger's
problem list cites the one-dimensional obstruction and asks whether the
optimal uniform prefix rate on the cube can fail in low dimensions for large
\(p\) \cite[Problem 47]{SteinerbergerOpenProblems}. We answer Problem 47 on
\([0,1]^d\) completely: such a sequence exists if and only if \(d\ge2\). In
fact, for every \(d\ge2\), one explicit dyadic digital sequence satisfies
\[
W_\infty(\mu_N,\lambda_d)\lesssim_d N^{-1/d}
\]
uniformly in \(N\). Notice that \(W_p\le W_\infty\) holds for every
\(1\le p<\infty\), so it immediately yields the optimal-order bound for the full
range \(1\le p\le\infty\). Finally, if we combine it with the one-dimensional obstruction, we settle Problem 47 on \(Q\) for every \(d\ge1\).

We return to this sequence in Section \ref{subsec:digital_sequence_winfty},
where we describe the dyadic counting property used in the proof.
Let us fix $d\ge2$ and set $b:=2^d$. For $a\in\{0,1,\dots,b-1\}$, we write its binary
expansion on $d$ bits as
\[
a=\sum_{j=1}^d \varepsilon_j(a)\,2^{j-1},\qquad \varepsilon_j(a)\in\{0,1\},
\]
and for an integer $m\ge0$, we write its base-$b$ expansion as
\[
m=\sum_{k=0}^\infty a_k(m)\,b^k,\qquad a_k(m)\in\{0,1,\dots,b-1\},
\]
where all but finitely many digits $a_k(m)$ are $0$.

\begin{definition}[Dyadic digital sequence]\label{def:digital_sequence_winfty}
We define $(x_n)_{n\ge1}\subset Q$ by
\begin{equation}\label{eq:digital_sequence_winfty}
x_{n}:=\Bigl(\sum_{k=0}^\infty \varepsilon_1(a_k(n-1))\,2^{-(k+1)},\ \dots,\
          \sum_{k=0}^\infty \varepsilon_d(a_k(n-1))\,2^{-(k+1)}\Bigr),\qquad n\ge1.
\end{equation}
\end{definition}
\noindent This implies that if $n-1$ has base-$2^d$ digits $a_0,a_1,\dots$, then the $j$th coordinate of $x_n$ is the dyadic number whose digits are
$\varepsilon_j(a_0),\varepsilon_j(a_1),\dots$.

\subsection{Main results}

\begin{theorem}[Prefix transport partition at optimal scale]\label{thm:main}
Let \(d\ge2\), and let \((x_n)_{n\ge1}\subset Q\) be the dyadic digital sequence
from \eqref{eq:digital_sequence_winfty}. For every \(N\in\N\), there exist
pairwise disjoint Borel sets \(A_1,\dots,A_N\subset Q\) such that
\[
\lambda_d(A_n)=\frac1N,\qquad
A_n\subset B(x_n,6\sqrt d\,N^{-1/d})\qquad(1\le n\le N),
\]
and
\[
\lambda_d\!\Bigl(Q\setminus\bigcup_{n=1}^N A_n\Bigr)=0.
\]
\end{theorem}

\begin{corollary}[Uniform optimal-order $W_\infty$ bound on the cube]\label{cor:winfty_main}
Under the hypotheses of Theorem \ref{thm:main}, for every \(N\in\N\),
\[
W_\infty\!\left(\frac1N\sum_{n=1}^N\delta_{x_n},\,\lambda_d\right)\le 6\sqrt d\,N^{-1/d}.
\]
\end{corollary}
\begin{proof}
It suffices to apply Lemma \ref{lem:partition_implies_winfty} to the partition obtained from Theorem \ref{thm:main}.
\end{proof}

\begin{corollary}[Uniform optimal-order Wasserstein bounds for all \(p\)]\label{cor:all_p}
Under the hypotheses of Corollary \ref{cor:winfty_main}, for every \(N\in\N\) and every
\(1\le p\le\infty\),
\[
W_p\!\left(\frac1N\sum_{n=1}^N\delta_{x_n},\,\lambda_d\right)\le 6\sqrt d\,N^{-1/d}.
\]
In particular, Steinerberger's Problem 47 has an affirmative answer on \(Q\) for every \(d\ge2\).
\end{corollary}
\begin{proof}
For every \(\pi\in\Pi(\mu_N,\lambda_d)\) and every \(1\le p<\infty\), we have
\[
\left(\int |x-y|^p\,\mathrm d\pi(x,y)\right)^{1/p}
\le \operatorname*{ess\,sup}_{(x,y)\sim\pi}|x-y|.
\]
We now take the infimum over \(\pi\) and obtain
\[
W_p(\mu_N,\lambda_d)\le W_\infty(\mu_N,\lambda_d).
\]
The case \(p=\infty\) is Corollary \ref{cor:winfty_main}. It now suffices to combine this
monotonicity with Corollary \ref{cor:winfty_main} to prove the claim.
\end{proof}

\begin{proposition}[Obstruction]\label{prop:1d_intro}
If $d=1$ and $1\le p\le\infty$, then no sequence $(x_n)\subset[0,1]$ can possibly satisfy
\[
W_p(\mu_N,\lambda_1)\le C_p\,N^{-1}\qquad\forall N\in\N
\]
for any finite constant $C_p$.
\end{proposition}

\noindent This could be directly deduced from Graham's lower bound on $\T^1$
\cite[Theorem 3]{Graham2019}. For completeness, we provide a short reduction in
Appendix \ref{app:one_dimensional_obstruction}.

\begin{remark}[Sharpness of the exponent]
The exponent $1/d$ is sharp. Indeed, for every $1\le p\le\infty$ and every probability measure
$\sigma$ supported on at most $N$ points in $Q$,
\[
W_p(\sigma,\lambda_d)\ge c_{p,d}\,N^{-1/d},
\]
for some constant $c_{p,d}>0$. A simple volumetric argument yields this lower bound.
For the analogous \(W_1\) estimate on the torus, see Brown--Steinerberger
\cite{BrownSteinerberger2020}. The cube estimate follows by the same
volumetric argument, and the extension to general \(p\) uses \(W_1\le W_p\).
\end{remark}

\subsection{Related work}
For each fixed $1\le p<\infty$, the asymptotic $N$-point quantisation problem
is classical. For compactly supported absolutely continuous measures, and more
generally under standard moment assumptions, the minimal $W_p$-error is of
order $N^{-1/d}$ as $N\to\infty$ (see, for instance, \cite{GrafLuschgy2000}; for
a broad survey of quantisation theory, consult \cite{GrayNeuhoff1998}). These
optimisers are chosen separately for each $N$, however, and do not provide a single infinite sequence
with uniform control over all prefixes. The deterministic prefix problem is also related to torus constructions. For instance, on the flat torus \(\T^d:=\R^d/\Z^d\), Brown--Steinerberger
\cite[Theorem 5]{BrownSteinerberger2020} proved a uniform optimal-order \(W_2\)
estimate for badly approximable Kronecker sequences in dimensions \(d\ge2\). For random empirical measures, finite-sample Wasserstein rates are known in general settings (see \cite{FournierGuillin2015,WeedBach2019}). Let $\widehat\lambda_N$
denote the empirical measure of $N$ i.i.d.\ $\lambda_d$-distributed points, then
the result by Boissard--Le Gouic \cite[Corollary 1.2]{BoissardLeGouic2014} implies that
\[
\mathbb E\,W_p(\widehat\lambda_N,\lambda_d)\lesssim N^{-1/d}
\]
on $[0,1]^d$ whenever $d>2p$.

\subsection{Proof strategy}
The proof of the main theorem is geometric. In Section \ref{sec:winfty_sharp}, we construct, for each $N$, an exact equal-mass transport partition of $Q$ into
$N$ Borel sets of measure $1/N$. Each is attached to a distinct point of
$\{x_1,\dots,x_N\}$ and contained in a ball of radius $O(N^{-1/d})$ centred at
that point. Lemma \ref{lem:partition_implies_winfty} then yields
Corollary \ref{cor:winfty_main}, and Corollary \ref{cor:all_p} follows directly from the
monotonicity $W_p\le W_\infty$. In Appendix \ref{app:one_dimensional_obstruction},
we prove the one-dimensional obstruction.

We use the standard notation $A\lesssim_{(\cdots)}B$ to mean
$A\le C\,B$ for a constant $C>0$ that depends only on the indicated parameters,
and $A\simeq_{(\cdots)}B$ when both $A\lesssim_{(\cdots)}B$ and
$B\lesssim_{(\cdots)}A$ hold.

\section{Transport partitions}\label{sec:winfty_sharp}

In this section, we work on the half-open cube \(Q=[0,1)^d\), which we identify with
\([0,1]^d\) up to a \(\lambda_d\)-null set. We build all partitions below
with half-open rectangles.
Since \(|Q|=1\), the normalised measure \(\lambda_d\) is the ordinary
\(d\)-dimensional Lebesgue measure on measurable subsets of \(Q\). For
measurable \(E\subset\mathbb R^d\), we write \(|E|\) for the ordinary Lebesgue
measure.

\subsection{A partition criterion for \texorpdfstring{$W_\infty$}{Winfty}}\label{subsec:winfty_partition_sufficient}

\begin{lemma}[A partition implies a $W_\infty$ bound]\label{lem:partition_implies_winfty}
Let $x_1,\dots,x_N\in Q$ and $\mu_N:=\frac1N\sum_{n=1}^N\delta_{x_n}$.
We also assume that for some $r\ge0$, there exist pairwise disjoint Borel sets $A_1,\dots,A_N\subset Q$ such that
\[
\lambda_d(A_n)=\frac1N,\qquad A_n\subset B(x_n,r)\quad(1\le n\le N),
\]
and $\lambda_d\!\bigl(Q\setminus\bigcup_{n=1}^N A_n\bigr)=0$.
It follows that $W_\infty^{(Q,|\cdot|)}(\mu_N,\lambda_d)\le r$.
\end{lemma}

\begin{proof}
By symmetry of $W_\infty$, it suffices to construct a coupling in $\Pi(\mu_N,\lambda_d)$. To this end, we define $\pi\in\Pi(\mu_N,\lambda_d)$ by
\[
\pi:=\sum_{n=1}^N \delta_{x_n}\otimes (\lambda_d|_{A_n}).
\]
This implies that $\pi$ has total mass $\sum_n \lambda_d(A_n)=1$. Its first marginal is
$\sum_n \lambda_d(A_n)\,\delta_{x_n}=\mu_N$, and its second marginal is
$\sum_n \lambda_d|_{A_n}=\lambda_d$, which follows from the fact that the $A_n$ partition $Q$ up to null sets.
Moreover, $\pi$ is supported on $\{(x_n,y):y\in A_n\}\subset\{(x,y):|x-y|\le r\}$, so
\[
W_\infty(\mu_N,\lambda_d)
\le \operatorname*{ess\,sup}_{(x,y)\sim\pi}|x-y|
\le r.
\]
This proves the claim.
\end{proof}

\subsection{Dyadic Cubes}\label{subsec:digital_sequence_winfty}

Recall the dyadic digital sequence from Definition \ref{def:digital_sequence_winfty}.
We set $b:=2^d$. This dyadic digital construction is closely related to standard digital
sequences in quasi-Monte Carlo theory. In other terms, it could also be understood as a
base-\(2\) digital construction obtained by distributing the binary digits of
\(n-1\) among the \(d\) coordinates. For more details on digital sequences, see
\cite[Chapter 4]{Niederreiter1992} and \cite[Chapter 4]{DickPillichshammer2010}.

For $\ell\in\N$ and a word $u=(u_0,\dots,u_{\ell-1})\in\{0,1,\dots,b-1\}^\ell$, we
define the level-$\ell$ dyadic cube
\begin{equation}\label{eq:dyadic_cube_u_winfty}
C_u:=\prod_{j=1}^d \Bigl[\ \sum_{k=0}^{\ell-1}\varepsilon_j(u_k)\,2^{-(k+1)}\ ,\
                         \sum_{k=0}^{\ell-1}\varepsilon_j(u_k)\,2^{-(k+1)}+2^{-\ell}\ \Bigr)\subset Q,
\end{equation}
which is half-open (this is required to remove boundary ambiguities). These $C_u$ are precisely the common dyadic cubes of side $2^{-\ell}$. For $\ell=0$, we interpret $u=\varnothing$. In this case, we have $C_\varnothing=Q$.

\begin{lemma}[Prefix counts in dyadic cubes]\label{lem:dyadic_counts_winfty}
Let $(x_n)$ be as in Definition \ref{def:digital_sequence_winfty}.
We also fix $\ell\in\N$ and $u\in\{0,\dots,b-1\}^\ell$.
For every $N\in\N$, we have
\begin{equation}\label{eq:counts_floor_ceil_winfty}
\#\{1\le n\le N:\ x_n\in C_u\}\ \in\ \Bigl\{\Bigl\lfloor \frac{N}{b^\ell}\Bigr\rfloor,\ \Bigl\lceil \frac{N}{b^\ell}\Bigr\rceil\Bigr\}.
\end{equation}
\end{lemma}

\begin{proof}
Let us write $m:=n-1$. By construction, $x_n\in C_u$ is equivalent to the requirement that the first $\ell$ base-$b$ digits of $m$ are the same as $u$, \textit{i.e.}, $a_k(m)=u_k$ for $k=0,\dots,\ell-1$.
In other terms,
\[
m\equiv r_u:=\sum_{k=0}^{\ell-1}u_k\,b^k\pmod{b^\ell}.
\]
It follows that $\{1\le n\le N:x_n\in C_u\}$ is in bijection with the integers $m\in\{0,1,\dots,N-1\}$ in the residue class $r_u\bmod b^\ell$.
If we now count a residue class in an initial interval, we obtain \eqref{eq:counts_floor_ceil_winfty}.
\end{proof}

\begin{remark}[Classical van der Corput counting]
Lemma \ref{lem:dyadic_counts_winfty} is the multidimensional analogue of the classical counting method for van der Corput-type sequences \cite{FaureKritzerPillichshammer2015}.
\end{remark}

\subsection{Near-dyadic mass splitting}\label{subsec:mass_splitting_winfty}

We now construct, for each $N$, a partition into $N$ equal-measure sets that lie within distance $\lesssim N^{-1/d}$ from the points $x_1,\dots,x_N$. Notice that in $d\ge2$, the necessary boundary shifts form a geometric series. There are two ideas behind each coordinate cut. On the one hand, we need some combinatorial control on how far the relevant child counts could deviate from exact halving, on the other, we need geometric control on the displacement of the cutting hyperplane. We now focus on the above issues in the following lemmas.

\begin{lemma}[Half-sums for $\{m,m+1\}$-valued multisets]\label{lem:half_sum_m_m1}
Let $B\in\N$ be even and let $c_1,\dots,c_B\in\Z$ satisfy $c_i\in\{m,m+1\}$ for some $m\in\Z$. Let $I\subset\{1,\dots,B\}$ with $|I|=B/2$. It follows that
\[
\Bigl|\sum_{i\in I}c_i - \frac12\sum_{i=1}^B c_i\Bigr|\ \le\ \frac{B}{4}.
\]
\end{lemma}

\begin{proof}
We first write $\sum_{i=1}^B c_i = Bm + r$, where $r\in\{0,1,\dots,B\}$ is the number of indices with $c_i=m+1$. We also set $\sum_{i\in I}c_i = (B/2)m + s$, where $s$ is the number of indices in $I$ with $c_i=m+1$.
We obtain:
\[
\sum_{i\in I}c_i-\frac12\sum_{i=1}^B c_i = s-\frac r2.
\]
Notice that $I$ has size $B/2$, so we obtain $\max\{0,r-B/2\}\le s\le \min\{r,B/2\}$.
If $r\le B/2$, then $s\in[0,r]$ and $|s-r/2|\le r/2\le B/4$.
Finally, if $r\ge B/2$, then $s\in[r-B/2,B/2]$ (with midpoint equal to $r/2$), so $|s-r/2|\le (B-r)/2\le B/4$.
\end{proof}

\begin{lemma}[Cut displacement in a rectangle]\label{lem:cut_displacement_winfty}
Let $R=\prod_{j=1}^d [a_j,b_j)\subset\R^d$ be an axis-parallel rectangle and fix an index $j\in\{1,\dots,d\}$.
Let
\[
\mathcal A_j:=\prod_{i\ne j}(b_i-a_i)
\]
be the $(d-1)$--dimensional cross-sectional area orthogonal to $e_j$.
For an arbitrary $V\in(0,|R|)$, there exists a unique $t\in(a_j,b_j)$ such that
\[
|R\cap\{x_j<t\}|=V.
\]
Moreover, if $t_{\mathrm{mid}}:=(a_j+b_j)/2$ is the midpoint of the $j$-th interval, then
\[
|t-t_{\mathrm{mid}}|
=\frac{|V-|R|/2|}{\mathcal A_j}.
\]
\end{lemma}

\begin{proof}
Since $|R\cap\{x_j<t\}|=\mathcal A_j\,(t-a_j)$ for $t\in(a_j,b_j)$ holds, we know that
the map $t\mapsto \mathcal A_j(t-a_j)$ is continuous and strictly increasing on $(a_j,b_j)$, so existence and uniqueness hold. The identity for $|t-t_{\mathrm{mid}}|$ follows by direct computation.
\end{proof}

Let us now fix $N\in\N$ and define
\begin{equation}\label{eq:def_L_winfty}
L:=\min\{\ell\in\{0,1,2,\dots\}:\ b^\ell\ge N\}.
\end{equation}
For each $\ell\in\{0,1,\dots,L-2\}$ and each word $u\in\{0,\dots,b-1\}^\ell$, define the prefix count
\[
M(u):=\#\{1\le n\le N:\ x_n\in C_u\}.
\]

\begin{lemma}[Choice of $L$]\label{lem:L_and_Mu_bounds}
Let $L$ be given by \eqref{eq:def_L_winfty}. We then obtain
\begin{equation}\label{eq:two_power_bounds_winfty}
2^{-L}\ \le\ N^{-1/d}\ \le\ 2\cdot2^{-L}.
\end{equation}
For every $\ell\in\{0,1,\dots,L-2\}$ and every word $u\in\{0,\dots,b-1\}^\ell$, we have
\begin{equation}\label{eq:M_u_bounds_winfty}
\Bigl\lfloor \frac{N}{b^\ell}\Bigr\rfloor\ \le\ M(u)\ \le\ \Bigl\lceil \frac{N}{b^\ell}\Bigr\rceil.
\end{equation}
If $N>b^2$, then
\begin{equation}\label{eq:Mu_Lminus2_positive}
M(u)\ \ge\ b\qquad\text{for all }|u|=L-2.
\end{equation}
\end{lemma}

\begin{proof}
Note that $L$ is minimal with $b^L\ge N$, so we have
\[
b^{L-1}<N\le b^L.
\]
It now suffices to apply $b=2^d$ and take \(-1/d\)-powers to obtain
\[
2^{-L}\le N^{-1/d}<2^{-(L-1)}=2\cdot 2^{-L},
\]
which proves \eqref{eq:two_power_bounds_winfty}. The estimate \eqref{eq:M_u_bounds_winfty} follows immediately from
Lemma \ref{lem:dyadic_counts_winfty}.

Finally, if $N>b^2$, then $L\ge3$ and still $b^{L-1}<N$. It follows that for every word
$u$ with $|u|=L-2$,
\[
\frac{N}{b^{L-2}}>b,
\]
and therefore
\[
M(u)\ge \Bigl\lfloor \frac{N}{b^{L-2}}\Bigr\rfloor\ge b
\]
by \eqref{eq:M_u_bounds_winfty}. This proves \eqref{eq:Mu_Lminus2_positive}.
\end{proof}

\begin{remark}[Induction at level $L-2$]
We terminate the induction at level \(L-2\) for geometric reasons, \textit{i.e.}, although \(b^{L-1}<N\) still implies that every level-\((L-1)\) dyadic cube contains at least one point, the drift estimate that we need for Lemma \ref{lem:one_step_split_winfty} is only established for \(\ell\le L-3\). We will use the bound \eqref{eq:Mu_Lminus2_positive} later to guarantee that the terminal rectangles \(R_u\) have positive volume.
\end{remark}

\paragraph{Controlled drift and rectangle partitions.}
Let us assume $L\ge3$. For $0\le \ell\le L-3$, we define
\begin{equation}\label{eq:Delta_S_def}
\Delta_\ell:=\frac{2^{2d-3}}{N}\,2^{\ell(d-1)},\qquad S_0:=0,\qquad S_{\ell+1}:=S_\ell+\Delta_\ell.
\end{equation}

\begin{lemma}[Controlled drift and splitting]\label{lem:one_step_split_winfty}
We assume $L\ge3$ and let $0\le \ell\le L-3$.
We further fix $u\in\{0,\dots,b-1\}^\ell$ and write
\[
C_u=\prod_{j=1}^d [A_j,A_j+2^{-\ell}),\qquad
R_u=\prod_{j=1}^d [a_j,b_j).
\]
Assume that
\[
\lambda_d(R_u)=\frac{M(u)}{N}
\]
and that
\[
|a_j-A_j|\le S_\ell,\qquad |b_j-(A_j+2^{-\ell})|\le S_\ell
\qquad(1\le j\le d).
\]
It follows that there exist half-open rectangles \(\{R_{u\ast v}\}_{v=0}^{b-1}\) such that:
\begin{enumerate}
\item[(i)] \(\{R_{u\ast v}\}_{v=0}^{b-1}\) partitions \(R_u\) up to null sets
\item[(ii)] \(\lambda_d(R_{u\ast v})=M(u\ast v)/N\) for every \(v\in\{0,\dots,b-1\}\)
\item[(iii)] If we write
\[
C_{u\ast v}=\prod_{j=1}^d [A_{u\ast v,j},A_{u\ast v,j}+2^{-(\ell+1)}),
\qquad
R_{u\ast v}=\prod_{j=1}^d [a_{u\ast v,j},b_{u\ast v,j}),
\]
we obtain
\[
|a_{u\ast v,j}-A_{u\ast v,j}|\le S_{\ell+1},\qquad
|b_{u\ast v,j}-(A_{u\ast v,j}+2^{-(\ell+1)})|\le S_{\ell+1}
\qquad(1\le j\le d).
\]
\end{enumerate}
\end{lemma}

\begin{proof}
For $v\in\{0,\dots,b-1\}$, we write $\varepsilon(v)=(\varepsilon_1(v),\dots,\varepsilon_d(v))\in\{0,1\}^d$, and for a bit-prefix $\eta=(\eta_1,\dots,\eta_k)\in\{0,1\}^k$, define
\[
M_u(\eta):=\sum_{\substack{v\in\{0,\dots,b-1\}:\\ \varepsilon_1(v)=\eta_1,\ \dots,\ \varepsilon_k(v)=\eta_k}} M(u\ast v),
\]
where \(u\ast v\) denotes the concatenation of the word \(u\) with the digit \(v\). We also set \(M_u(\varnothing):=M(u)\), so that \(M_u(\eta,0)+M_u(\eta,1)=M_u(\eta)\).

We shall first consider the target volumes. Note that since $\ell+1\le L-2$ holds, we have $b^{\ell+1}<N$ and thus $\lfloor N/b^{\ell+1}\rfloor\ge1$. By Lemma \ref{lem:dyadic_counts_winfty}, every child count $M(u\ast v)\ge1$, so every partial sum $M_u(\eta,0)$ and $M_u(\eta,1)$ is also $\ge1$. We conclude that all target volumes below are strictly between $0$ and the parent volume, so Lemma \ref{lem:cut_displacement_winfty} applies at every step.

Let us now set $R_u^{\varnothing}:=R_u$. For $k=1,2,\dots,d$ and each $\eta\in\{0,1\}^{k-1}$, let $t=t(u,\eta,k)$ be the unique number in the $k$-th coordinate interval of $R_u^{(\eta)}$ such that
\[
\lambda_d\bigl(R_u^{(\eta)}\cap\{x_k<t\}\bigr)=\frac{M_u(\eta,0)}{N}.
\]
Define
\[
R_u^{(\eta,0)}:=R_u^{(\eta)}\cap\{x_k<t\},\qquad
R_u^{(\eta,1)}:=R_u^{(\eta)}\cap\{x_k\ge t\}.
\]
Notice that by Lemma \ref{lem:cut_displacement_winfty}, the cut location $t$ is uniquely determined, and
\[
\lambda_d\bigl(R_u^{(\eta,0)}\bigr)=\frac{M_u(\eta,0)}{N},\qquad
\lambda_d\bigl(R_u^{(\eta,1)}\bigr)=\frac{M_u(\eta,1)}{N}.
\]
After the $d$ cuts, we define
\[
R_{u\ast v}:=R_u^{(\varepsilon_1(v),\dots,\varepsilon_d(v))},\qquad v\in\{0,\dots,b-1\}.
\]
It follows that $\{R_{u\ast v}\}_v$ is a partition of $R_u$ into half-open rectangles up to null sets and satisfies (ii), as required.

Let us now focus on the control of cut positions. We define the minimal parent side length as follows:
\[
\sigma_\ell:=2^{-\ell}-2S_\ell.
\]
Observe that $\ell\le L-3$ implies $b^{\ell+2}\le b^{L-1}<N$, so we have
\[
S_\ell=\sum_{r=0}^{\ell-1}\Delta_r
\le \frac{2^{2d-3}}{N}\sum_{r=0}^{\ell-1}2^{r(d-1)}
\le \frac{2^{2d-3}}{N}\cdot \frac{2^{\ell(d-1)}}{2^{d-1}-1}
\le 2^{-\ell-3}.
\]
It holds that $\sigma_\ell\ge 2^{-\ell}-2^{-\ell-2}=\frac34\,2^{-\ell}$, and also
\[
\Delta_\ell=\frac{2^{2d-3}}{N}2^{\ell(d-1)}\le 2^{-\ell-3}\le \frac{\sigma_\ell}{6}.
\]

We now claim that \textit{every} cut at level $\ell$ differs from the midpoint of the interval being cut by at most $\Delta_\ell$. Our strategy is to establish this by induction on the cut index $k=1,\dots,d$. To this end, fix $k\in\{1,\dots,d\}$ and $\eta\in\{0,1\}^{k-1}$, and consider the cut that splits $R_u^{(\eta)}$ in the $k$-th coordinate.
Note that only the first $k-1$ coordinates have been cut so far, which implies that the $k$-th coordinate interval of $R_u^{(\eta)}$ is still the parent interval $[a_k,b_k)$.
Let $t$ be the cut location in coordinate $k$ and let
\[
t_{\mathrm{mid}}:=\frac{a_k+b_k}{2}
\]
be the midpoint of this interval.
By Lemma \ref{lem:cut_displacement_winfty}, we obtain
\[
|t-t_{\mathrm{mid}}|
=\frac{\Bigl|\frac{M_u(\eta,0)}{N}-\frac12\frac{M_u(\eta)}{N}\Bigr|}{\mathcal A_k}
=\frac{1}{N}\cdot\frac{\bigl|M_u(\eta,0)-\frac12 M_u(\eta)\bigr|}{\mathcal A_k},
\]
where $\mathcal A_k$ is the cross-sectional area of $R_u^{(\eta)}$ that is orthogonal to $e_k$.

Let us now establish the numerator bound.
Among the children $u\ast v$ that are consistent with $\eta$, there are
\[
B:=2^{d-k+1}
\]
choices of $v$. If we split by the $k$-th bit, we then divide them into two subfamilies of size $B/2$. Moreover, by Lemma \ref{lem:dyadic_counts_winfty} at level $\ell+1$, each $M(u\ast v)\in\{m,m+1\}$ with $m=\lfloor N/b^{\ell+1}\rfloor$ independent of $v$.
It now suffices to apply Lemma \ref{lem:half_sum_m_m1} to obtain
\[
\bigl|M_u(\eta,0)-\tfrac12 M_u(\eta)\bigr|\le \frac{B}{4}=2^{d-k-1}.
\]

To establish the denominator bound, we fix a previously cut coordinate $j<k$.
Observe that when the cut in direction $e_j$ was completed, the current $j$-interval was still the parent interval $[a_j,b_j)$, which has length at least $\sigma_\ell$.
Recall that we made the cut within $\Delta_\ell$ of its midpoint, so each output child
$j$-interval has length at least
\[
\frac{b_j-a_j}{2}-\Delta_\ell \ge \frac{\sigma_\ell}{2}-\Delta_\ell \ge \frac{\sigma_\ell}{3},
\]
since $\Delta_\ell\le \sigma_\ell/6$.
This implies that every previously cut coordinate interval has length at least $\sigma_\ell/3$. Every coordinate interval that has not yet been cut equals the corresponding parent interval and so has length at least $\sigma_\ell$.
We conclude that the cross-sectional area orthogonal to $e_k$ satisfies
\[
\mathcal A_k\ \ge\ \Bigl(\frac{\sigma_\ell}{3}\Bigr)^{k-1}\sigma_\ell^{d-k}
=\frac{\sigma_\ell^{d-1}}{3^{k-1}}.
\]

It now suffices to combine the above numerator and denominator bounds to obtain
\[
|t-t_{\mathrm{mid}}|
\le \frac{1}{N}\cdot 2^{d-k-1}\cdot \frac{3^{k-1}}{\sigma_\ell^{d-1}}
\le \frac{1}{N}\cdot 2^{d-k-1}3^{k-1}\cdot \Bigl(\frac{4}{3}\Bigr)^{d-1}2^{\ell(d-1)},
\]
where we use $\sigma_\ell\ge \frac34\,2^{-\ell}$.
Note that the factor $2^{d-k-1}3^{k-1}$ increases in $k$ and its maximum over $k\in\{1,\dots,d\}$ is attained at $k=d$, that is
\[
2^{d-k-1}3^{k-1}\le 2^{-1}3^{d-1}.
\]
It follows that
\[
|t-t_{\mathrm{mid}}|
\le \frac{1}{N}\cdot \frac12\,3^{d-1}\cdot \Bigl(\frac{4}{3}\Bigr)^{d-1}2^{\ell(d-1)}
=\frac{2^{2d-3}}{N}\,2^{\ell(d-1)}
=\Delta_\ell.
\]

At this point, we must still establish how to control the child endpoints.
We fix $k\in\{1,\dots,d\}$ and $\eta\in\{0,1\}^{k-1}\), and let $t$ be the cut location used to split $R_u^{(\eta)}$ in the $k$-th coordinate.
Notice that the first $k-1$ cuts are active only in coordinates $1,\dots,k-1$, the $k$-th coordinate interval of $R_u^{(\eta)}$ is still the parent interval $[a_k,b_k)$. We infer that its midpoint is
\[
t_{\mathrm{mid}}=\frac{a_k+b_k}{2}.
\]
In particular, the intermediate dyadic descendant determined by the previously chosen
bits $\eta$ still has $k$-th coordinate interval $[A_k,A_k+2^{-\ell})$. This is due to the fact that the previous dyadic cuts are also active only in coordinates $1,\dots,k-1$. It follows that the dyadic cutting hyperplane is
\[
x_k=A_k+2^{-(\ell+1)}.
\]

Recall that we assumed endpoint bounds at level $\ell$. From those, we now obtain
\[
|a_k-A_k|\le S_\ell,\qquad |b_k-(A_k+2^{-\ell})|\le S_\ell.
\]
It follows that
\begin{align*}
\left|t_{\mathrm{mid}}-\left(A_k+2^{-(\ell+1)}\right)\right|
&=\frac12\left|(a_k-A_k)+\bigl(b_k-(A_k+2^{-\ell})\bigr)\right|\\
&\le \frac12\Bigl(|a_k-A_k|+|b_k-(A_k+2^{-\ell})|\Bigr)\\
&\le S_\ell.
\end{align*}
We now combine this with the midpoint estimate that proved above. From
\[
|t-t_{\mathrm{mid}}|\le \Delta_\ell,
\]
we obtain
\[
\left|t-\left(A_k+2^{-(\ell+1)}\right)\right|
\le |t-t_{\mathrm{mid}}|
+\left|t_{\mathrm{mid}}-\left(A_k+2^{-(\ell+1)}\right)\right|
\le \Delta_\ell+S_\ell
= S_{\ell+1}.
\]

Let us now fix $v\in\{0,\dots,b-1\}$. For each $j\in\{1,\dots,d\}$, let $t_j$ denote the cut location used at the $j$-th stage along the branch determined by the previously chosen bits
\[
\bigl(\varepsilon_1(v),\dots,\varepsilon_{j-1}(v)\bigr).
\]
It follows that the $j$-th coordinate interval of $R_{u\ast v}$ is
\[
[a_j,t_j)\quad\text{if }\varepsilon_j(v)=0,
\qquad\text{and}\qquad
[t_j,b_j)\quad\text{if }\varepsilon_j(v)=1.
\]
On the other hand, the $j$-th coordinate interval of $C_{u\ast v}$ is
\[
[A_j,A_j+2^{-(\ell+1)})\quad\text{or}\quad
[A_j+2^{-(\ell+1)},A_j+2^{-\ell}).
\]

Observe that the old endpoints satisfy
\[
|a_j-A_j|\le S_\ell\le S_{\ell+1},\qquad
|b_j-(A_j+2^{-\ell})|\le S_\ell\le S_{\ell+1},
\]
and the new endpoint fulfils
\[
|t_j-(A_j+2^{-(\ell+1)})|\le S_{\ell+1}
\]
by the estimate we establish (note that we applied it with $k=j$ and the appropriate preceding bit-prefix). It follows that every endpoint of $R_{u\ast v}$ differs from the corresponding endpoint of $C_{u\ast v}$
by at most $S_{\ell+1}$. In other terms, if we write
\[
C_{u\ast v}=\prod_{j=1}^d [A_{u\ast v,j},A_{u\ast v,j}+2^{-(\ell+1)}),
\qquad
R_{u\ast v}=\prod_{j=1}^d [a_{u\ast v,j},b_{u\ast v,j}),
\]
we obtain
\[
|a_{u\ast v,j}-A_{u\ast v,j}|\le S_{\ell+1},\qquad
|b_{u\ast v,j}-(A_{u\ast v,j}+2^{-(\ell+1)})|\le S_{\ell+1}
\qquad(1\le j\le d).
\]
We conclude that (iii) holds for every child rectangle.
\end{proof}

\begin{proposition}[Dyadic rectangle partition]\label{prop:dyadic_rectangle_partition}
Let us assume $L\ge3$. For each $\ell\in\{0,1,\dots,L-2\}$, there exists a family $\{R_u\}_{|u|=\ell}$ of half-open rectangles in $Q$ such that:
\begin{enumerate}
\item[(i)] $\{R_u:|u|=\ell\}$ partitions $Q$ up to $\lambda_d$--null sets
\item[(ii)] $\lambda_d(R_u)=M(u)/N$ for every $|u|=\ell$
\item[(iii)] If we write $C_u=\prod_{j=1}^d [A_{u,j},A_{u,j}+2^{-\ell})$ and $R_u=\prod_{j=1}^d [a_{u,j},b_{u,j})$,
we obtain the endpoint bounds
\begin{equation}\label{eq:endpoint_bounds}
|a_{u,j}-A_{u,j}|\le S_\ell,\qquad |b_{u,j}-(A_{u,j}+2^{-\ell})|\le S_\ell
\qquad(1\le j\le d).
\end{equation}
\end{enumerate}
\end{proposition}

\begin{proof}
We establish the claim by induction on $\ell$. 

\noindent\textit{Base case: $\ell=0$.}
Let us take $R_\varnothing:=Q=C_\varnothing$. Note that (i) and (ii) hold since $M(\varnothing)=N$, and \eqref{eq:endpoint_bounds} holds with $S_0=0$.

\noindent\textit{Induction step: $\ell\to\ell+1$ for $\ell\le L-3$.}
Let us assume that we have already constructed rectangles $\{R_u:|u|=\ell\}$ that satisfy (i)--(iii). For each word $u$ of length $\ell$, we apply Lemma \ref{lem:one_step_split_winfty} to the pair \((C_u,R_u)\). We then obtain half-open child rectangles \(\{R_{u\ast v}\}_{v=0}^{b-1}\) that partition \(R_u\) up to null sets, have the correct masses \(M(u\ast v)/N\), and satisfy the level-\((\ell+1)\) endpoint bounds. If we now take the union over all words \(u\) of length \(\ell\), we obtain a family
\(\{R_w\}_{|w|=\ell+1}\) with properties (i)--(iii) at level \(\ell+1\). This
completes the induction.
\end{proof}

\begin{lemma}[Drift up to level $L-2$]\label{lem:SL_bound}
If we assume $L\ge3$, then
\[
S_{L-2}\ \le\ \frac{2^{d-1}}{2^{d-1}-1}\,2^{-L}\ \le\ 2\cdot 2^{-L}.
\]
\end{lemma}

\begin{proof}
By definition, we have
\[
S_{L-2}=\sum_{\ell=0}^{L-3}\Delta_\ell
=\frac{2^{2d-3}}{N}\sum_{\ell=0}^{L-3}2^{\ell(d-1)}
\le \frac{2^{2d-3}}{N}\cdot \frac{2^{(L-2)(d-1)}}{2^{d-1}-1}.
\]
From $b^{L-1}<N$ and $b=2^d$, we have $N>2^{d(L-1)}$, and thus
\[
S_{L-2}
\le \frac{2^{2d-3}}{2^{d(L-1)}}\cdot \frac{2^{(L-2)(d-1)}}{2^{d-1}-1}
=\frac{2^{d-1}}{2^{d-1}-1}\,2^{-L}.
\]
This proves the lemma.
\end{proof}

\subsection{Proof of the allocation theorem}\label{subsec:winfty_sharp_cube}

\begin{proof}[Proof of Theorem \ref{thm:main}]
The case $N\le b^2$ follows immediately. Let us choose an arbitrary partition of $Q$ into pairwise disjoint Borel sets $A_1,\dots,A_N$ of measure $1/N$. From
\[
6\sqrt d\,N^{-1/d}\ge 6\sqrt d\,b^{-2/d}=\frac32\sqrt d\ge \diam(Q),
\]
we infer that each $A_n$ is contained in $B(x_n,6\sqrt d\,N^{-1/d})$.
This establishes the theorem for $N\le b^2$. 

For the rest of the proof, we assume that $N>b^2$, so $L\ge3$.
Let $\{R_u:|u|=L-2\}$ be the level-$(L-2)$ rectangle partition from
Proposition \ref{prop:dyadic_rectangle_partition}. By \eqref{eq:Mu_Lminus2_positive}, we know that each $R_u$ has positive volume. For every $u$ with $|u|=L-2$, we write
\[
R_u=\prod_{j=1}^d [a_{u,j},b_{u,j}).
\]
$\lambda_d(R_u)=M(u)/N$ implies that there exist unique numbers
\[
a_{u,1}=s_{u,0}<s_{u,1}<\cdots<s_{u,M(u)}=b_{u,1}
\]
such that
\[
(s_{u,r}-s_{u,r-1})\prod_{j=2}^d (b_{u,j}-a_{u,j})=\frac1N
\qquad(r=1,\dots,M(u)).
\]
This implies that the rectangles
\[
[s_{u,r-1},s_{u,r})\times\prod_{j=2}^d [a_{u,j},b_{u,j}),
\qquad r=1,\dots,M(u),
\]
are pairwise disjoint half-open rectangles of $\lambda_d$--measure $1/N$ whose union is $R_u$. If we now repeat this procedure for all $u$, we obtain exactly $N=\sum_{|u|=L-2}M(u)$ pairwise disjoint rectangles of $\lambda_d$--measure $1/N$ that cover $Q$ up to null sets. We shall enumerate them as $\{A_1,\dots,A_N\}$.

Note that for each $u$ with $|u|=L-2$, the cube $C_u$ contains exactly $M(u)$ of the points $x_1,\dots,x_N$. We choose an arbitrary bijection between these points and the $M(u)$ sub-rectangles of $R_u$. Once we have relabelled the sub-rectangles as $A_1,\dots,A_N$, we may assume that $A_n\subset R_u$ whenever $x_n\in C_u$.

We fix such a pair $(A_n,x_n)$ with $A_n\subset R_u$ and $x_n\in C_u$.
By the endpoint bounds \eqref{eq:endpoint_bounds} at level $L-2$,
we know that each coordinate interval of $R_u$ differs from that of $C_u$ by at most $S_{L-2}$. For every $y=(y_1,\dots,y_d)\in R_u$, we then define $z=(z_1,\dots,z_d)\in\overline{C_u}$ by letting $z_j$ be the metric projection of $y_j$ onto the interval
\[
[A_{u,j},A_{u,j}+2^{-(L-2)}]\qquad(1\le j\le d),
\]
where $C_u=\prod_{j=1}^d [A_{u,j},A_{u,j}+2^{-(L-2)})$.
From
\[
|a_{u,j}-A_{u,j}|\le S_{L-2},\qquad
|b_{u,j}-(A_{u,j}+2^{-(L-2)})|\le S_{L-2},
\]
we infer that every $y_j\in[a_{u,j},b_{u,j})$ satisfies $|y_j-z_j|\le S_{L-2}$.
It follows that
\[
|y-z|\le \sqrt d\,S_{L-2}.
\]
Since $x_n\in C_u\subset\overline{C_u}$ holds, we also have
\[
|y-x_n|
\le |y-z|+|z-x_n|
\le \sqrt d\,S_{L-2}+\diam(C_u)
=\sqrt d\,S_{L-2}+\sqrt d\,2^{-(L-2)}.
\]
Note that every \(y\in A_n\) satisfies
\[
|y-x_n|\le \sqrt d\bigl(S_{L-2}+2^{-(L-2)}\bigr).
\]
By the proof of Lemma \ref{lem:SL_bound}, we obtain the strict estimate
\[
S_{L-2}<\frac{2^{d-1}}{2^{d-1}-1}\,2^{-L}\le 2\cdot 2^{-L}.
\]
Since $2^{-(L-2)}=4\cdot 2^{-L}$ holds, it follows that
\[
\sqrt d\bigl(S_{L-2}+2^{-(L-2)}\bigr)\ <\ 6\sqrt d\,2^{-L}\ \le\ 6\sqrt d\,N^{-1/d}.
\]
Finally, we conclude that
\[
A_n\subset B(x_n,6\sqrt d\,N^{-1/d})\qquad\forall n.
\]
This proves Theorem \ref{thm:main}.
\end{proof}

\appendix
\section{The one-dimensional obstruction}\label{app:one_dimensional_obstruction}

\begin{proof}[Proof of Proposition \ref{prop:1d_intro}]
Let \(\lambda_{\T^1}\) denote the Haar probability measure on \(\T^1\), and let $f:[0,1]\to\T^1$ be the quotient map $f(t):=t\bmod 1$. For an arbitrary Borel
probability measure \(\mu\) on \([0,1]\), let \(f_\#\mu\) denote the
pushforward of \(\mu\) under \(f\). It follows that $f$ is $1$--Lipschitz from $([0,1],|\cdot|)$ to $(\T^1,\rho)$, where
\[
\rho(x,y):=\min_{m\in\Z}|x-y-m|
\]
is the flat metric on $\T^1$, and it also satisfies $f_\#\lambda_1=\lambda_{\T^1}$.

We set $\nu_N:=f_\#\mu_N$. Indeed, if \(\pi\in\Pi(\mu_N,\lambda_1)\), then \((f,f)_\#\pi\in\Pi(f_\#\mu_N,f_\#\lambda_1)\). Moreover, \(f\) is \(1\)-Lipschitz, so we have
\(\rho(f(x),f(y))\le |x-y|\), and thus
\[
W_p^{(\T^1,\rho)}(\nu_N,\lambda_{\T^1})
= W_p^{(\T^1,\rho)}(f_\#\mu_N,f_\#\lambda_1)
\le W_p^{([0,1],|\cdot|)}(\mu_N,\lambda_1).
\]
By Graham's theorem \cite[Theorem 3]{Graham2019}, we conclude that for every sequence on $\T^1$, there exist infinitely many $N$ with the property that
\[
W_1^{(\T^1,\rho)}(\nu_N,\lambda_{\T^1})\gtrsim \frac{\sqrt{\log N}}{N}.
\]
Since $W_1\le W_p$ holds for every $1\le p\le\infty$, a uniform bound
\[
W_p(\mu_N,\lambda_1)\le C_p\,N^{-1}\qquad\forall N\in\N
\]
would imply
\[
W_1^{(\T^1,\rho)}(\nu_N,\lambda_{\T^1})\lesssim N^{-1}
\]
for all $N$, but this leads to a contradiction.
\end{proof}

\end{document}